\documentclass[graybox]{svmult}

\usepackage{amsfonts,amsmath,amstext,latexsym,amssymb,physics}
\usepackage{mathptmx}       
\usepackage{helvet}         
\usepackage{courier}        
\usepackage{type1cm}        
\usepackage{makeidx}         
\usepackage{graphicx}        
\usepackage{multicol}        
\usepackage[bottom]{footmisc}

\makeindex             

\begin{document}

\title*{Commutants in crossed product algebras for piecewise constant functions on the real line}
  \titlerunning{Commutants in crossed products} 
\author{Alex Behakanira Tumwesigye, Johan Richter, Sergei Silvestrov }
 \authorrunning{A. B. Tumwesigye, J. Richter \& S. Silvestrov}
\institute{Alex Behakanira Tumwesigye \at Department of Mathematics, College of Natural Sciences, Makerere University, Box 7062, Kampala, Uganda. \\ \email{alexbt@cns.mak.ac.ug}
\and Johan Richter
\at Department of Mathematics and Natural Sciences, Blekinge Institute of Technology, SE-37179 Karlskrona, Sweden. \\ \email{johan.richter@bth.se}
\and Sergei Silvestrov
\at Division of Applied Mathematics, School of Education, Culture and Communication, M\"alardalen University, Box 883, 72123 V\"aster{\aa}s, Sweden. \\ \email{sergei.silvestrov@mdh.se}
}

%
%
\maketitle

\abstract*{In this paper we consider commutants in crossed product algebras, for algebras of piece-wise constant functions on the real line acted on by the group of integers $\mathbb{Z}$. The algebra of piece-wise constant functions does not separate points of the real line, and interplay of the action with separation properties of the points or subsets of the real line by the function algebra become essential for many properties of the crossed product algebras and their subalgebras. In this article, we deepen investigation of properties of this class of crossed product algebras and interplay with dynamics of the actions. We describe the commutants and changes in the commutants in the crossed products for the canonical generating commutative function subalgebras of the algebra of piece-wise constant functions with common jump points when arbitrary number of jump points are added or removed in general positions, that is when corresponding constant value sets partitions of the real line change, and we give complete characterization of the set difference between commutants for the increasing sequence of subalgebras in crossed product algebras for algebras of functions that are constant on sets of a partition when partition is refined.}

\abstract{In this paper we consider commutants in crossed product algebras, for algebras of piece-wise constant functions on the real line acted on by the group of integers $\mathbb{Z}$. The algebra of piece-wise constant functions does not separate points of the real line, and interplay of the action with separation properties of the points or subsets of the real line by the function algebra become essential for many properties of the crossed product algebras and their subalgebras. In this article, we deepen investigation of properties of this class of crossed product algebras and interplay with dynamics of the actions. We describe the commutants and changes in the commutants in the crossed products for the canonical generating commutative function subalgebras of the algebra of piece-wise constant functions with common jump points when arbitrary number of jump points are added or removed in general positions, that is when corresponding constant value sets partitions of the real line change, and we give complete characterization of the set difference between commutants for the increasing sequence of subalgebras in crossed product algebras for algebras of functions that are constant on sets of a partition when partition is refined.}




\section{Introduction}
An important direction of investigation for any class of non-commutative algebras and rings, is the description of commutative subalgebras and commutative subrings. This is because such a description allows one to relate representation theory, non-commutative properties, graded structures, ideals and subalgebras, homological and other properties of non-commutative algebras to spectral theory, duality, algebraic geometry and topology naturally associated with commutative algebras. In representation theory, for example, semi-direct products or crossed products play a central role in the construction and classification of representations using the method of induced representations. When a non-commutative algebra is given, one looks for a subalgebra such that its representations can be studied and classified more easily and such that the whole algebra can be decomposed as a crossed product of this subalgebra by a suitable action.\par
When one has found a way to present a non-commutative algebra as a crossed product of a commutative subalgebra by some action on it, then it is important to know whether the subalgebra is maximal commutative, or if not, to find a maximal commutative subalgebra containing the given subalgebra. This maximality of a commutative subalgebra and related properties of the action are intimately related to the description and classification of representations of the non-commutative algebra.\par
Some work has been done in this direction \cite{Silvestrov1TRS,LiTRS,TomiyamaTRS} where the interplay between topological dynamics of the action on one hand and the algebraic property of the commutative subalgebra in the $C^*-$crossed product algebra $C(X)\rtimes \mathbb{Z}$ being maximal commutative on the other hand are considered. In \cite{Silvestrov1TRS}, an explicit description of the (unique) maximal commutative subalgebra containing a subalgebra $\mathcal{A}$ of $\mathbb{C}^X$ is given. In \cite{OinertTRS}, properties of commutative subrings and ideals in non-commutative algebraic crossed products by arbitrary groups are investigated and a description of the commutant of the base coefficient subring in the crossed product ring is given. More results on commutants in crossed products  and dynamical systems can be found in  \cite{SvenssonTRS,CarlsenTRS} and the references therein.

In this article, we consider commutants in crossed product algebras for algebras of piece-wise constant functions on the real line. In \cite{AlexTRS}, a description of the maximal commutative subalgebra  (commutant) of the crossed product algebra of the said algebra with $\mathbb{Z}$ was given for the case where we have $N$ fixed jumps and in \cite{Alex1TRS}, a comparison of commutants for an increasing sequence of algebras with a finite number of jumps added into one of the partition intervals was done. Here, we treat a more general case, whereby starting with an algebra $\mathcal{A}$ of piecewise constant functions with $N$ fixed jump points, we add a finite number of jumps, say $m$ arbitrarily and consider the algebra $\mathcal{A}_S$ of piecewise constant functions with $N+m$ jumps. We derive a condition for the algebras $\mathcal{A}$ and $\mathcal{A}_S$ to be invariant under a bijection $\sigma:\mathbb{R}\to \mathbb{R}$ and compare the commutants $\mathcal{A}'$ and $\mathcal{A}_S'.$
\section{Definitions and a preliminary result }
Let  $\mathcal{A}$ be any commutative  algebra. Using the notation in \cite{Silvestrov1TRS}, let $\phi:\mathcal{A}\rightarrow \mathcal{A}$ be any algebra automorphism on $\mathcal{A}$ and define $$\mathcal{A}\rtimes_{\phi}\mathbb{Z}:=\{f:\Bbb Z\rightarrow \mathcal{A}: f(n)=0\mbox{ except for a finite number of }n\}.$$
Then \cite{Silvestrov1TRS}  $\mathcal{A}\rtimes_{\phi}\mathbb{Z}$ is an associative $\mathbb{C}-$ algebra with respect to point-wise addition, scalar multiplication and multiplication defined by \emph{ twisted convolution}, $*$ as follows;
$$(f* g)(n)=\sum_{k\in \mathbb{Z}}f(k)\phi^k(g(n-k)),$$ where $\phi^k$ denotes the $k-$fold composition of $\phi$ with itself for positive $k$ and we use the obvious definition for $k\leq 0$
\begin{definition}\label{defn1TRS}
$\mathcal{A}\rtimes_{\phi}\mathbb{Z}$ as described above is called the crossed product algebra of $\mathcal{A}$ and $\mathbb{Z}$ under $\phi.$
\end{definition}
A useful and convenient way of working with $\mathcal{A}\rtimes_{\phi}\mathbb{Z}$, is to  write elements $f,g\in \mathcal{A}\rtimes_{\phi}\mathbb{Z}$ in the form $f=\sum_{n\in \Bbb Z}f_n\delta^n$ and $g=\sum_{n\in \Bbb Z}g_m\delta^m$ where $f_n=f(n),~g_m=g(m)$ and
$$\delta^n(k)=\begin{cases}1, ~~\mbox{if } k=n\\
0, ~~\mbox{if }k\neq n  .
\end{cases}$$ In the sum $\sum_{n\in \mathbb{Z}}f_n\delta^n$, we implicitly assume that $f_n=0$ except for a finite number of $n.$ Addition and scalar multiplication are canonically defined by the usual pointwise operations  and multiplication is determined by the relation
\begin{equation}\label{eqn1TRS}
(f_n\delta^n)*(g_m\delta^m)=f_n\phi^n(g_m)\delta^{n+m}
\end{equation} where $m,n\in \mathbb{Z}$ and $f_n,g_m\in \mathcal{A}.$
\begin{definition}\label{defn2TRS}
By the commutant $\mathcal{A}'$ of $\mathcal{A}$ in $\mathcal{B}\rtimes_{\phi}\mathbb{Z}$, where $A\subseteq B$, we mean $$\mathcal{A}':=\{f\in \mathcal{B}\rtimes_{\phi}\mathbb{Z}:fg=gf \mbox{ for every }g\in \mathcal{A}\}.$$
\end{definition}

Frequently the algebra $\mathcal{B}$ in the previous definition will be clear from context. 

It has been proven \cite{Silvestrov1TRS} that the commutant $\mathcal{A}'$ in $\mathcal{A}\rtimes_{\phi}\mathbb{Z}$  is commutative and thus, is the unique maximal commutative subalgebra containing $\mathcal{A}$.

We shall also use

\subsection{Automorphisms induced by bijections}
Now let $X$ be any set and $\mathcal{A}$ an algebra of complex valued functions on $X$. Let $\sigma:X\rightarrow X$ be any bijection such that  $\mathcal{A}$ is invariant under $\sigma$ and $\sigma^{-1},$ that is for every $h\in \mathcal{A},\ h\circ \sigma \in \mathcal{A}$ and $h\circ \sigma^{-1}\in \mathcal{A}.$
 Then $(X, \sigma)$ is a discrete dynamical system and $\sigma$ induces an automorphism $\tilde{\sigma}:\mathcal{A}\rightarrow \mathcal{A}$ defined by,
 \begin{equation}\label{eqn2TRS}
 \tilde{\sigma}(f)=f\circ \sigma^{-1}.
 \end{equation}
 Observe that since $\tilde{\sigma}=f\circ \sigma^{-1},$ then $\tilde{\sigma}^2(f)=\tilde{\sigma}(f\circ \sigma^{-1})=(f\circ \sigma^{-1})\circ \sigma^{-1}=f\circ \sigma^{-2}$ and in general $\tilde{\sigma}^n(f)=f\circ \sigma^{-n}$ for all $n\in \mathbb{Z}.$

  In \cite{AlexTRS}, a description of  the commutant of $\mathcal{A}'$ in the crossed product algebra $\mathcal{A}\rtimes_{\tilde{\sigma}}\Bbb Z$ for the  case where $\mathcal{A}$ is the algebra of functions that are constant on the sets of a partition was given. Below are some definitions and results that will be important in our study. The proofs of the theorems can be found in \cite{AlexTRS} and the references in there.
\begin{definition}
For any nonzero $n\in \Bbb Z,$ let
\begin{equation} \label{eqn3TRS}
Sep_{\mathcal{A}}^n(X):=\{x\in X~|~\exists\ h\in \mathcal{A}~:~h(x)\neq \tilde{\sigma}^n(h)(x)\},\\
\end{equation}
\end{definition}
The following theorem has been proven in \cite{Silvestrov1TRS}.
\begin{theorem}\label{thm1TRS}
The unique maximal commutative subalgebra of $\mathcal{A}\rtimes_{\tilde{\sigma}}\Bbb Z$ that contains $\mathcal{A}$ is precisely the set of elements
\begin{equation}\label{eqn4TRS}
\mathcal{A}'=\left\{\sum_{n\in \Bbb Z}f_n\delta^n~|~\mbox{for all }n\in \Bbb Z:~f_n|Sep_{\mathcal{A}}^n(X)\equiv 0\right\} \end{equation}
\end{theorem}

\section{Commutants in crossed product algebras for piecewise constant functions on the real line}
Our aim is to compare commutants for algebras of piecewise constant functions defined on a real line when jump points are added arbitrarily. First let's state some results already known.\par
Let $\mathcal{A}$ be the algebra of piece-wise constant functions $f:\Bbb R\rightarrow \Bbb R$ with $N$ fixed jumps at points $t_1,t_2,\cdots,t_N.$ Partition $\Bbb R$ into $N+1$ intervals $I_0,I_1,\cdots,I_N$ where $I_{\alpha}=(t_{\alpha},t_{\alpha +1})$ with $t_0=-\infty$ and $t_{N+1}=\infty.$  By looking at jump points as intervals of zero length, we can write $\mathbb{R}=\bigcup\limits_{\alpha=0}^{2N} I_{\alpha}$ where $I_{\alpha}$ is as described above for $\alpha=0,1,\cdots N$ and $I_{N+\alpha}=\{t_{\alpha} \}$ for $\alpha=1,2,\cdots,N.$ Then every $h\in \mathcal{A}$ can be written as
\begin{equation}\label{eqn5TRS}
h(x)=\sum_{\alpha=0}^{2N}a_{\alpha}\chi_{I_{\alpha}}(x),
\end{equation} where $\chi_{I_{\alpha}}$ is the characteristic function of $I_{\alpha}$ and $a_{\alpha}$ are some constants.\par
 Let $\sigma :\Bbb R \rightarrow \Bbb R$ be any bijection on $\Bbb R$ such that $\mathcal{A}$ is invariant under $\sigma$ . The following lemma gives the necessary and sufficient conditions for $(\Bbb R, \sigma)$ to be a discrete dynamical system.
\begin{lemma}\label{lem1TRS}
The algebra $\mathcal{A}$ is invariant under both $\sigma$ and $\sigma^{-1}$ if and only if the following conditions hold.
\begin{enumerate}
\item $\sigma$ (and $\sigma^{-1}$) maps  each jump point $t_k,~k=1,\cdots,N$ onto another jump point.
\item $\sigma$ maps every interval $I_{\alpha},~\alpha=0,1,\cdots N$ bijectively onto any of the other intervals $I_0,I_1 \cdots I_N.$
\end{enumerate}
\end{lemma}
\begin{remark}\label{Remark1TRS}
It is important to note that our algebras are isomorphic to certain function algebras on finite sets, ie certain finite dimensional algerbras. Because of the connection with other types of function algebras on the real line, explained in the introduction, we prefer to phrase things in terms of the algebra of piecewise constant functions. For more details of this isomorphism see  \cite[Remark 3.2]{AlexTRS}.
\end{remark}

 Let $\sigma :\Bbb R \rightarrow \Bbb R$ be any bijection on $\Bbb R$ such that $\mathcal{A}$ is invariant under $\sigma$, $\tilde{\sigma}:\mathcal{A}\rightarrow \mathcal{A}$ be the automorphism on $\mathcal{A}$ induced by $\sigma$, as given by \eqref{eqn2TRS} and consider the crossed product algebra $\mathcal{A}\rtimes_{\tilde{\sigma}}\mathbb{Z}$. The following proposition gives the description of $Sep_{\mathcal{A}}^n(\Bbb R)$ for any $n\in \Bbb Z.$
\begin{proposition}\label{Propn1TRS}
Let $\mathcal{A}$ be the algebra of piecewise constant functions on the real line with $N$ fixed jumps as described above and let $\sigma :\Bbb R \rightarrow \Bbb R$ be any bijection on $\Bbb R$ such that $\mathcal{A}$ is invariant under $\sigma$ (and $\sigma^{-1}$). Let $\tilde{\sigma}:\mathcal{A}\rightarrow \mathcal{A}$ be the automorphism on $\mathcal{A}$ induced by $\sigma$.
 Then for every $n\in \Bbb Z,$
\begin{equation}
{Sep_{\mathcal{A}}^n}(\Bbb R)=\bigcup_{k\nmid n}C_k, \label{eqn6TRS}
\end{equation}
where \begin{align}\label{eqn7TRS}
C_k& :=\left\{x\in \mathbb{R}~|~k\mbox{ is the smallest positive integer such that  } \exists\ I_{\alpha},\text{ such that } x,\sigma^k(x)\in I_{\alpha}\right.\nonumber \\
&\qquad \left. {}\text{ for some }\alpha=0,1,\cdots,2N\right\}.
\end{align}
\end{proposition}

\begin{theorem}\label{thm2TRS}
Let $\mathcal{A}$ be the algebra of piece-wise constant functions $f:\Bbb R\rightarrow \Bbb R$ with $N$ fixed jumps at points $t_1,\cdots,t_N$ as described above, $\sigma :\Bbb R \rightarrow \Bbb R$ be any bijection on $\Bbb R$ such that $\mathcal{A}$ is invariant under both $\sigma$ and $\sigma^{-1}$ and let $\tilde{\sigma}:\mathcal{A}\rightarrow \mathcal{A}$ be the automorphism on $\mathcal{A}$ induced by $\sigma$.
 Then the unique maximal commutative subalgebra of $\mathcal{A}\rtimes_{\tilde{\sigma}}\Bbb Z$ that contains $\mathcal{A}$ is given by,
\[\mathcal{A}'=\left\{\sum_{n\in \Bbb Z}f_n\delta^n\ |\ f_n\equiv 0\ \text{on }C_k \text{ for some $n$ such that }k\nmid  n \right\},\]
where $C_k$ is as defined in \eqref{eqn7TRS}.
\end{theorem}
\begin{proof}
It has been proven that the unique maximal commutative subalgebra $\mathcal{A}'$, of $\mathcal{A}\rtimes_{\tilde{\sigma}}\Bbb Z$ that contains $\mathcal{A}$ is given by
\begin{equation}\label{eqn8TRS}
\mathcal{A}'=\left\{\sum_{n\in \Bbb Z}f_n\delta^n~|~\mbox{for all }n\in \Bbb Z:~f_n|Sep_{\mathcal{A}}^n(X)\equiv 0\right\}. \end{equation}
Therefore, comparing \eqref{eqn6TRS} and \eqref{eqn8TRS}, we have;
\begin{equation}\label{eqn9TRS}
\mathcal{A}'=\left\{\sum_{n\in \Bbb Z}f_n\delta^n\ |\ f_n\equiv 0\ \text{on }C_k \text{ for some  $n$ such that }k\nmid  n \right\}.\end{equation}
 \end{proof}
 \begin{remark}\label{Remark2TRS}
 \end{remark}
 The crossed product algebra $\mathcal{A}\rtimes_{\tilde{\sigma}}\mathbb{Z}$ is a strongly $\mathbb{Z}-$graded algebra but the commutant $\mathcal{A}'$ is $\mathbb{Z}-$graded but not strongly $\mathbb{Z}-$graded as can be seen from the following observation. \par
 Observe that we can write $\mathcal{A}'$ as
 $$
 \mathcal{A}'=\bigoplus_{n\in\mathbb{Z}} \mathcal{A}_n'
 $$
 where $$\mathcal{A}_n'=\{f_n\delta^n\ :\ f_n=0 \text{ on }Sep_{\mathcal{A}}^n(\mathbb{R})\}.$$
 Therefore, if $f_n\delta^n\in \mathcal{A}_n'$ and $f_m\delta^m\in \mathcal{A}_m',$ then
 $$
 f_n\delta^n * f_m\delta^m=f_n\tilde{\sigma}^n(f_m)\delta^{n+m}
 $$ will be zero on $Sep_{\mathcal{A}}^n(\mathbb{R}).$ Since $Sep_{\mathcal{A}}^{n+m}\subsetneq Sep_{\mathcal{A}}^n(\mathbb{R}),$ we conclude that $\mathcal{A}'$ is not strongly $\mathbb{Z}-$graded.
\section{Jump points added arbitrarily}
Let $\mathcal{A}$ be the algebra of piece-wise constant functions $f:\Bbb R\rightarrow \Bbb R$ with $N$ fixed jumps at points $t_1,t_2,\cdots,t_N.$ Partition $\Bbb R$ into $N+1$ intervals $I_0,I_1,\cdots,I_N$ where $I_{\alpha}=(t_{\alpha},t_{\alpha +1})$ with $t_0=-\infty$ and $t_{N+1}=\infty.$  By looking at jump points as intervals of zero length, we can write $\mathbb{R}=\cup I_{\alpha}$ where $I_{\alpha}$ is as described above for $\alpha=0,1,\cdots N$ and $I_{N+\alpha}=\{t_{\alpha} \}$ for $\alpha=1,2,\cdots,N.$ Suppose $S=\{s_1,\cdots,s_m\}$ is a set of points in $\mathbb{R}$ and let $\mathcal{A}_S=\mathcal{A}_{t_1,t_2,\cdots,t_N,s_1,\cdots,s_m}$ be an algebra of piecewise constant functions on $\mathbb{R}$ with at most $N+m$ fixed jumps at points $t_1,\cdots,t_N,s_1,\cdots,s_m.$ We want to do the following.
\begin{enumerate}
\item Derive conditions under which $\mathcal{A}$ and $\mathcal{A}_S$ are both invariant under a bijection $\sigma:\mathbb{R}\to \mathbb{R}.$
\item Derive an expression for $Sep_{\mathcal{A}_S}^n(\mathbb{R})$ for any $n\in \mathbb{Z}$, comparing it with $Sep_{\mathcal{A}}^n(\mathbb{R})$ and find the commutant $\mathcal{A}_{S}'.$
\end{enumerate}
\subsection{A condition for invariance }
Since $\mathcal{A}$ is a subalgebra of $\mathcal{A}_S,$ invariance of both algebras under $\sigma$ ensures that the crossed product algebra $\mathcal{A}\rtimes_{\tilde{\sigma}}\mathbb{Z}$ is a subalgebra of the crossed product algebra $\mathcal{A}_S\rtimes_{\tilde{\sigma}}\mathbb{Z}$ and therefore we can compare the respective commutants $\mathcal{A}'$ and $\mathcal{A}_S',$ provided that we understand $A'$ to mean the commutant of $A$ in $\mathcal{A}_S\rtimes_{\tilde{\sigma}}\mathbb{Z}$.
The following Lemma gives a condition under which the algebras $\mathcal{A}$ and $\mathcal{A}_S$ are both invariant under a bijection $\sigma:\mathbb{R}\to \mathbb{R}.$
\begin{lemma}\label{lem2TRS}
Suppose that the jump points $s_1,\cdots,s_m$ are added into the intervals $I_{\alpha_1},\cdots,I_{\alpha_m}$ respectively, that is, $s_i\in I_{\alpha_i}$ for $i=1,\cdots,n.$ Let $\sigma:\mathbb{R}\to\mathbb{R}$ be a bijection such that $\mathcal{A}$ and $\mathcal{A}_S$ are both invariant under $\sigma$ (and $\sigma^{-1}$). Then
\begin{equation}\label{eqn10TRS}
\sigma\left(\cup_{i=1}^m I_{\alpha_i}\right)=\cup_{i=1}^m I_{\alpha_i}
\end{equation}
\end{lemma}
\begin{proof}
Suppose $\sigma(I_{\alpha_i})=I_{\beta}$ for some $\beta\notin \{\alpha_1,\cdots,\alpha_m\}.$ Since $s_i\in I_{\alpha_i}$ is a jump point and $\mathcal{A}_S$ is invariant under $\sigma$, then $\sigma(s_i)$ must be a jump point. Therefore $\sigma(s_i)\in \{t_1,t_2,\cdots,t_N\}.$ Since $\mathcal{A}$ is invariant under $\sigma,$ then $\sigma(\{t_1,\cdots,t_n\})=\{t_1,\cdots,t_N\}.$ Therefore, we have that $\sigma(\{t_1,\cdots,t_n,s_i\})=\{t_1,\cdots,t_N\},$ which contradicts bijectivity of $\sigma.$\end{proof}
\section{Finitely many jump points added}
Suppose a finite number of jump points, $s_1,s_2,\cdots, s_m$ are added into intervals, say $I_{\alpha_1},I_{\alpha_2},\cdots,I_{\alpha_r}$ with $r\leqslant m.$ Then either these jump points are added into the same interval, say $I_{\alpha_0}$ or into different intervals. As mentioned before, a detailed description of the commutant for the case when jump points are added into the same interval was done in \cite{Alex1}. Therefore we concentrate on the case when jump points are added into different intervals.
\subsection{Jumps added into different intervals}
 Suppose that the jump points $s_1,s_2,\cdots,s_m$ are added into distinct intervals, say, $I_{\alpha_1},I_{\alpha_2},\cdots,I_{\alpha_r},$ with $r\leqslant m.$ As before, let $\mathcal{A}$ be the algebra of piecewise constant functions with $N$ fixed jumps at $t_1, \cdots, t_N$ and let $\mathcal{A}_S$ denote the algebra of piecewise constant functions with $N+m$ fixed jumps at points $t_1,t_2,\cdots,t_N,s_1,\cdots,s_m.$  Suppose $\sigma:\mathbb{R}\to \mathbb{R}$ is a bijection on $\mathbb{R}$ such the algebras $\mathcal{A}$ and $\mathcal{A}_S$ are both invariant under $\sigma.$ Then by Lemma \ref{lem2TRS},
 $
 \sigma\left(\cup_{i=1}^r I_{\alpha_i}\right)=\cup_{i=1}^r I_{\alpha_i}.
 $
 However, we have the following Lemma that gives the connection between the number of jump points that can be added into intervals that belong to one cycle.
 \begin{lemma}\label{lem3TRS}
 Suppose $p$ jump points are added into an interval $I_{\alpha}$ and $q$ jump points are added into an interval $I_{\beta}.$ If $\sigma :\mathbb{R}\to \mathbb{R}$ is a bijection on $\mathbb{R}$ such that $\mathcal{A}_S$ is invariant under $\sigma$ and $\sigma(I_{\alpha})=I_{\beta},$ then $p=q.$
 \end{lemma}
 \begin{proof}
 If there are $p$ jump points in $I_{\alpha}$ and $q$ jump points in $I_{\beta}$ and $\sigma:\mathbb{R}\to \mathbb{R}$ is a bijection on $\mathbb{R}$ such that $\mathcal{A}_S$ is invariant under $\sigma,$ then by Lemma \ref{invariant1}, $\sigma$ maps jump points in $I_{\alpha}$ to jump points in $I_{\beta}.$ Since $\sigma$ is a bijection, then the number of jump points in $I_{\alpha}$ must be equal to the number of jump points in $I_{\beta}.$  Therefore, $p=q.$
 \end{proof}
 \subsection{A comparison of the commutants}
 Let\begin{align*}
C_k& :=\left\{x\in \mathbb{R}~|~k\mbox{ is the smallest positive integer such that  } \exists\ I_{\alpha},\text{ such that } x,\sigma^k(x)\in I_{\alpha}\right.\nonumber \\
&\qquad \left. {}\text{ for some }\alpha=0,1,\cdots,2N\right\}.
\end{align*}
Observe that such $C_k$ consist of intervals, say $I_{\alpha_1},\cdots,I_{\alpha_k}$ that are mapped cyclically onto each other. Lemma \ref{lem3TRS} says that if we add $p$ jump points into one of these intervals, then we should add $p$ jump points into each of these intervals. Also, note that since we are adding $p$ jump points, each of the intervals $I_{\alpha_i},\ i=1,\cdots,k$ will be subdivided into $2p+1$ new subintervals of the form $I_{\alpha_i}^j,$ where
$I_{\alpha_i}^j=(s_{j-1}^i,s_j^i),\ j=1,\cdots, p+1$ with $s_0^i=t_{\alpha_i},\ s_{p+1}^i=t_{\alpha_i+1},\ i=1,\cdots,k$ and $\ I_{\alpha_{i}}^{p+j}=\{s_j^i\},\ j=1\cdots, p$. Also, let
\begin{align*}
\tilde{C}_k& :=\left\{x\in I_{\alpha_i}~|~k\mbox{ is the smallest positive integer such that  } \exists\ I_{\alpha_i}^j,\text{ such that } x,\sigma^k(x)\in I_{\alpha_i}^j\right.\nonumber \\
&\qquad \left. {}\text{ for some }i=0,\cdots,2p,\ j=1,\cdots,p\right\}.
\end{align*}

\begin{lemma}\label{lem4TRS}
Let $x\in I_{\alpha_i}$ where $I_{\alpha_i}\subset C_k$. Suppose we add $p$ jump points as described above. Then $x\in I_{\alpha_i}^j\subset \tilde{C}_{kl}$ for some $l\in \{1,2,\cdots,p+1\}.$
\end{lemma}
\begin{proof}
By invariance of $\mathcal{A}$ under $\sigma,$ we have that $\sigma$ maps the intervals $I_{\alpha_i},\ i=1,\cdots,k$ bijectively onto each other and since we are adding $p$ jump points into each interval $I_{\alpha_i}$, then each of these intervals is subdivided into $2p+1$ subintervals as described before. By invariance of $\mathcal{A}_S$ under $\sigma,$ we have that $\sigma$ maps each of the jump points $s_j^i\in I_{\alpha_i}$ onto another jump point. Since $\sigma^k(I_{\alpha_i})=I_{\alpha_i}$ for each $i=1,\cdots,k$, then each jump point belongs to $\tilde{C}_{kl}$ for some $l\in \{1,2,\cdots,p\}.$\par
Now consider the subintervals $I_{\alpha_i}^j=(s_{j-1}^i,s_j^i),\ i=1,2,\cdots,p+1.$ We know that each $I_{\alpha_i}$ is divided into $p+1$ intervals of this form. Furthermore, invariance of $\mathcal{A}_S$ under $\sigma$ implies that $\sigma$ maps each of these intervals bijectively onto each other. Since $\sigma^k(I_{\alpha_i})=I_{\alpha_i},$ then each $I_{\alpha_i}^j\subset \tilde{C}_{kl}$ for some $l\in \{1,2,\cdots,p+1\}.$
\end{proof}
From Lemma \ref{lem4TRS}, it can be seen that for those $C_k$ which contain intervals where we add $p$ jump points,
\begin{equation}\label{eqn11TRS}
C_k=\bigcup_{1\leqslant l\leqslant p+1}\tilde{C}_{kl}.
\end{equation}
Using this and the fact that for any integers $k,l,n$ with $k\neq 0,$ $kl\mid n$ if and only if $l\mid {n\over{k}},$ we give the description of $Sep_{\mathcal{A}_S}^n(\mathbb{R})$ and the commutant in the following theorem, whose proof is a direct consequence of Lemma \ref{lem4TRS} and equations \eqref{eqn6TRS} and \eqref{eqn9TRS}.
\begin{theorem}
Suppose we add $p$ jump points into each of the intervals in $C_k.$ Then
\begin{equation}\label{eqn12TRS}
Sep_{\mathcal{A}_S}^n(\mathbb{R})=\begin{cases} Sep_{\mathcal{A}}^n(\mathbb{R})& \text{ if } k\nmid n\\
Sep_{\mathcal{A}}^n(\mathbb{R})\bigcup\left(\bigcup\limits_{l=1\ :\ l\nmid{n\over k}}^{p+1}\tilde{C}_{kl}\right)& \text{ if } k\mid n
\end{cases}
\end{equation}
and the set difference of the commutants is given by
\begin{equation}\label{eqn13TRS}
\mathcal{A}_S'=\mathcal{A}'\setminus \left\{\sum_{n\in \mathbb{Z}}f_n\delta^n\ :\ f_n\neq 0 \text{ on }\tilde{C}_{kl}\text{ for some $l$ such that } l\nmid{n\over k}\right\}.
\end{equation}
\end{theorem}
From \eqref{eqn11TRS}, it can be seen that those $C_k$ which contain intervals can be decomposed as a union of $\tilde{C}_{kl}.$ In the next Theorem, we state a necessary condition for subintervals of a given  interval $I_{\alpha_i}\subset C_k$  to belong to $\tilde{C}_{kl}$.
\begin{theorem}\label{thm4TRS}
For each $l\in \{1,2,\cdots,p+1\}$ let $\pi(l)$ denote the number of subintervals (of an interval $I_{\alpha_i}\subset C_k$) that belong to $\tilde{C}_{kl}.$ Then
\begin{enumerate}
\item $l$ divides $\pi(l)$ for all $l$ and
\item $\sum\limits_{l=1}^{p+1}\pi(l)=p+1.$
\end{enumerate}
If the two above conditions are satisfied then $\pi$ counts the number of subintervals in $\tilde{C}_{kl}$ for some $\sigma$ and some choice of $p$ jump points.
\end{theorem}
\begin{proof}
Recall that, since we are adding $p$ jump points into each of the intervals $I_{\alpha_i},i=1,\cdots,k,$ each of these intervals is subdivided into $p+1$ subintervals (excluding the jump points). Therefore, from the definition of $\pi(l),$ $$\sum\limits_{l=1}^{p+1}\pi(l)=p+1.$$ Observe that an interval $I_{\alpha_i}^j\subset \tilde{C}_{kl}$ if and only if $\sigma^{kl}(I_{\alpha_i}^j)=I_{\alpha_i}^j.$ This means that there are $kl-1$ other intervals which, together with $I_{\alpha_i}^j,$ are permuted by $\sigma.$ That is, $\tilde{C}_{kl}$ contains cycles of subintervals (can be more than one cycle), of length $kl$ that are equally distributed into the intervals $I_{\alpha_i},i=1,\cdots,k.$ Therefore, $l\mid \pi(l).$
\end{proof}
\section{An example with two jump points added}
Suppose two jump points are added. Then these are either added into the same interval $I_{\alpha_0},$ say, or  they are added into two different intervals, say, $I_{\alpha_1}$ and $I_{\alpha_2}.$ We treat the two cases below.
\subsection{Jump points added into the same interval}
Suppose the two jump points are added into the same interval, say, $I_{\alpha_0}.$ Then this interval will be partitioned into three new subintervals which, together with the jump points, yields a new partition as follows.  $I_{\alpha_0}^1=(t_{\alpha_0},s_1),\ I_{\alpha_0}^2=(s_1,s_2),\ I_{\alpha_0}^3=(s_2,t_{\alpha_0+1}),\ I_{\alpha_0}^4=\{s_1\}$ and $I_{\alpha_0}^5=\{s_2\}.$
From Lemma \ref{lem2TRS}, we have that $\mathcal{A}_S$ is invariant under a bijection $\sigma:\mathbb{R}\to \mathbb{R}$ if $\sigma(I_{\alpha_0})=I_{\alpha_0}.$ Therefore $I_{\alpha_0}\not\subset Sep_{\mathcal{A}}^n(\mathbb{R})$ for any $n\in \mathbb{Z}.$
Let \begin{align}\label{eqn14TRS}
C_k& :=\left\{x\in \mathbb{R}~|~k\mbox{ is the smallest positive integer such that  } \exists\ I_{\alpha},\text{ such that } x,\sigma^k(x)\in I_{\alpha}\right.\nonumber \\
&\qquad \left. {}\text{ for some }\alpha=0,1,\cdots,2N\right\}.
\end{align}
and let
\begin{align}\label{eqn15TRS}
\tilde{C}_k& :=\left\{x\in I_{\alpha_0}~|~k\mbox{ is the smallest positive integer such that  } \exists\ I_{\alpha_0}^j,\text{ such that } x,\sigma^k(x)\in I_{\alpha_0}^j\right.\nonumber \\
&\qquad \left. {}\text{ for some }j=1,\cdots,5\right\}.
\end{align}
Then it is easily seen that, for every $n\in \mathbb{Z}$,
$$
Sep_{\mathcal{A}_S}^n(\mathbb{R})=Sep_{\mathcal{A}}^n(\mathbb{R})\bigcup\left(\cup_{k\nmid n}\tilde{C}_k\right).
$$
We treat the different cases below.
\subsubsection{$\sigma(I_{\alpha_0}^j)=I_{\alpha_0}^j$ for all $j=1,2,\cdots,5$}\label{case1TRS}
In this case, $I_{\alpha_0}^j\subset \tilde{C}_1$ for all $j=1,2,\cdots,5$ and hence  $I_{\alpha_0}^j\not\subset Sep_{\mathcal{A}_S}^n(\mathbb{R})$ for any $n\in \mathbb{Z.}$ Therefore, for any $n\in \mathbb{Z},$
$$
Sep_{\mathcal{A}_S}^n(\mathbb{R})=Sep_{\mathcal{A}}^n(\mathbb{R})
$$ and hence,
$$\mathcal{A}_S'=\mathcal{A}'$$
\subsubsection{$\sigma(I_{\alpha_0}^1)=I_{\alpha_0}^2,\ \sigma(I_{\alpha_0}^2)=I_{\alpha_0}^1$ and $\sigma(I_{\alpha_0}^j)=I_{\alpha_0}^j,\ j=3,4,5$}
In this case; $I_{\alpha_0}^1,I_{\alpha_0}^2\subset \tilde{C}_2$ and $I_{\alpha_0}^3,I_{\alpha_0}^4,I_{\alpha_0}^5\subset \tilde{C}_1.$  Therefore $I_{\alpha_0}^1,I_{\alpha_0}^2\subset Sep_{\mathcal{A}_S}^n(\mathbb{R})$ for every odd $n\in \mathbb{Z}.$ It should be noted that $\sigma(I_{\alpha_0})=I_{\alpha_0}$, therefore $I_{\alpha_0}\not\subset Sep_{\mathcal{A}}^n(\mathbb{R})$ for every $n\in \mathbb{Z}.$ We deduce that, for every $n\in \mathbb{Z},$
$$
Sep_{\mathcal{A}_S}^n(\mathbb{R})=\begin{cases} Sep_{\mathcal{A}}^n(\mathbb{R})& \text{ if $n$ is even }\\
Sep_{\mathcal{A}}^n(\mathbb{R})\bigcup \left(I_{\alpha_0}^1\cup I_{\alpha_0}^2\right)& \text{ if $n$ is odd.} \end{cases}
$$
Therefore, the commutant $\mathcal{A}_S'$ is given by;
$$
\mathcal{A}_S'=\mathcal{A}'\setminus \left\{\sum_{n\in \mathbb{Z}}f_n\delta^n\ |\ _{2n+1}\neq 0 \text{ on }I_{\alpha_0}^1\cup I_{\alpha_0}^2\right\}.
$$
The following cases produce similar results.
\begin{enumerate}
\item $\sigma(I_{\alpha_0}^1)=I_{\alpha_0}^3,\ \sigma(I_{\alpha_0}^3)=I_{\alpha_0}^1$ and $\sigma(I_{\alpha_0}^j)=I_{\alpha_0}^j$ for $j=2,4,5.$
\item $\sigma(I_{\alpha_0}^2)=I_{\alpha_0}^3,\ \sigma(I_{\alpha_0}^3)=I_{\alpha_0}^2$ and $\sigma(I_{\alpha_0}^j)=I_{\alpha_0}^j$ for $j=1,4,5.$
\item $\sigma(I_{\alpha_0}^4)=I_{\alpha_0}^5,\ \sigma(I_{\alpha_0}^5)=I_{\alpha_0}^4$ and $\sigma(I_{\alpha_0}^j)=I_{\alpha_0}^j$ for $j=1,2,3.$
\end{enumerate}
\subsubsection{$\sigma(I_{\alpha_0}^1)=I_{\alpha_0}^2,\ \sigma(I_{\alpha_0}^2)=I_{\alpha_0}^3,\ \sigma(I_{\alpha
_0}^3)=I_{\alpha_0}^1$ and $\sigma(I_{\alpha_0}^j)=I_{\alpha_0}^j$ for $j=4,5$}
In this case, $I_{\alpha_0}^4,I_{\alpha_0}^5\subset \tilde{C}_1$ and $I_{\alpha_0}^1,I_{\alpha_0}^2,I_{\alpha_0}^3\subset \tilde{C}_3.$ Therefore $I_{\alpha_0}^1,I_{\alpha_0}^2,I_{\alpha_0}^3\subset Sep_{\mathcal{A}_S}^n(\mathbb{R})$ for any $n\in \mathbb{Z}$ such that $3\nmid n.$ Therefore for every $n\in \mathbb{Z}$,
$$
Sep_{\mathcal{A}_S}^n(\mathbb{R})=\begin{cases} Sep_{\mathcal{A}}^n(\mathbb{R})& \text{ if $3\mid n$ }\\
Sep_{\mathcal{A}}^n(\mathbb{R})\bigcup \left(I_{\alpha_0}^1\cup I_{\alpha_0}^2\cup{I_{\alpha_0}^3}\right)& \text{ if $3\nmid n$}
\end{cases}
$$
and hence,
$$
\mathcal{A}_S'=\mathcal{A}'\setminus \left\{\sum_{n\in \mathbb{Z}}f_n\delta^n\ |\ _{n}\neq 0 \text{ on }I_{\alpha_0}^1\cup I_{\alpha_0}^2\cup I_{\alpha_0}^3 \text{ for all $n$ such that }3\nmid n\right\}.
$$
\subsubsection{$\sigma(I_{\alpha_0}^1)=I_{\alpha_0}^2,\ \sigma(I_{\alpha_0}^2)=I_{\alpha_0}^3,\ \sigma(I_{\alpha
_0}^3)=I_{\alpha_0}^1,\ \sigma(I_{\alpha_0}^4)=I_{\alpha_0}^5$ and $\sigma(I_{\alpha_0}^5)=I_{\alpha_0}^4$}
In this case, subintervals together with jump points are mapped cyclically by $\sigma.$\\ Therefore $I_{\alpha_0}^4,I_{\alpha_0}^5\subset \tilde{C}_2$ and $I_{\alpha_0}^1,I_{\alpha_0}^2,I_{\alpha_0}^3\subset \tilde{C}_3.$ It follows that $I_{\alpha_0}^1,I_{\alpha_0}^2,I_{\alpha_0}^3\subset Sep_{\mathcal{A}_S}^n(\mathbb{R})$ for any $n\in \mathbb{Z}$ such that $3\nmid n$ and $I_{\alpha_0}^4,I_{\alpha_0}^5\subset Sep_{\mathcal{A}_S}^n(\mathbb{R})$ for every odd $n\in \mathbb{Z}.$ Therefore for every $n\in \mathbb{Z}$,
$$
Sep_{\mathcal{A}_S}^n(\mathbb{R})=\begin{cases} Sep_{\mathcal{A}}^n(\mathbb{R})\cup\left(I_{\alpha_0}^4\cup I_{\alpha_0}^5\right)& \text{ if $n$ is odd and $3\mid n$ }\\
Sep_{\mathcal{A}}^n(\mathbb{R})\bigcup \left(I_{\alpha_0}^1\cup I_{\alpha_0}^2\cup{I_{\alpha_0}^3}\right)& \text{ if $3\nmid n$ } \\ Sep_{\mathcal{A}}^n(\mathbb{R})& \text{ otherwise}\end{cases}
$$
and hence,
\begin{align*}
\mathcal{A}_S'=\mathcal{A}'\setminus\left( \left\{\sum_{n\in \mathbb{Z}}f_n\delta^n\ |\ _{n}\neq 0 \text{ on }I_{\alpha_0}^1\cup I_{\alpha_0}^2\cup I_{\alpha_0}^3 \text{ for some $n$ such that }3\nmid n\right\}\bigcup\right. \\
 \left. {} \left\{\sum_{n\in \mathbb{Z}}f_n\delta^n\ |\ f_{2n+1}\neq 0 \text{ on }I_{\alpha_0}^4\cup I_{\alpha_0}^5\right\}\right).
\end{align*}
\subsection{Jump points added into different intervals}
Suppose the jump points are added into two different intervals, says $I_{\alpha_1}=(t_{\alpha_1},t_{\alpha_1+1})$ and $I_{\alpha_2}=(t_{\alpha_2},t_{\alpha_2+1})$, that is $t_{\alpha_1}<s_1<t_{\alpha_1+1}$ and $t_{\alpha_2}<s_2<t_{\alpha_2+1}.$ By Lemma \ref{lem2TRS},
$
\sigma\left(I_{\alpha_1}\cup I_{\alpha_2}\right)=I_{\alpha_1}\cup I_{\alpha_2}.
$
Suppose each of the intervals $I_{\alpha_i}$ is subdivided into subintervals $I_{\alpha_i}^j,\ j=1,2,3,$ where $I_{\alpha_i}^1=(t_{\alpha_i},s_i),\ I_{\alpha_i}^2=(s_i,t_{\alpha_i+1})$ and $I_{\alpha_i}^3=\{s_i\}.$ Again, let
\begin{align}\label{eqn16TRS}
C_k& :=\left\{x\in \mathbb{R}~|~k\mbox{ is the smallest positive integer such that  } \exists\ I_{\alpha},\text{ such that } x,\sigma^k(x)\in I_{\alpha}\right.\nonumber \\
&\qquad \left. {}\text{ for some }\alpha=0,1,\cdots,2N\right\}
\end{align}
and let
\begin{align}\label{eqn17TRS}
\tilde{C}_k& :=\left\{x\in I_{\alpha_i}~|~k\mbox{ is the smallest positive integer such that  } \exists\ I_{\alpha_i}^j,\text{ such that } x,\sigma^k(x)\in I_{\alpha_i}^j\right.\nonumber \\
&\qquad \left. {}\text{ for some }j=1,\cdots,3\right\}.
\end{align}
Then, again it can easily be seen that, for every $n\in \mathbb{Z},$
$$
Sep_{\mathcal{A}_S}^n(\mathbb{R})=Sep_{\mathcal{A}}^n(\mathbb{R})\bigcup\left(\cup_{k\nmid n}\tilde{C}_k\right)
$$
In this case we have two important scenarios.
\begin{enumerate}
\item $\sigma(s_i)=s_i$ for all $i=1,2$ or
\item $\sigma(s_1)=s_2$ and $\sigma(s_2)=s_1.$
\end{enumerate}
These can be further subdivided into other cases and we treat these in the following subsections.
\subsubsection{$\sigma(s_i)=s_i,\ i=1,2$ and $\sigma(I_{\alpha_i}^j)=I_{\alpha_i}^j$ for all $j=1,2,3$}\label{Case2TRS}
If $\sigma(s_1)=s_1$ and $\sigma(s_2)=s_2,$ then $\sigma(I_{\alpha_1})=I_{\alpha_1}$ and $\sigma(I_{\alpha_2})=I_{\alpha_2}.$ If in addition $\sigma(I_{\alpha_i}^j)=I_{\alpha_i}^j$ for all $i=1,2$ and $j=1,2,3,$ then all the new subintervals belong to $\tilde{C}_1$ and nothing changes in $Sep_{\mathcal{A}}^n(\mathbb{R}).$ That is $Sep_{\mathcal{A}_S}^n(\mathbb{R})=Sep_{\mathcal{A}}^n(\mathbb{R})$ and hence
$$\mathcal{A}_S'=\mathcal{A}'.$$
\subsubsection{$\sigma(s_i)=s_i$ for all $i=1,2$, $\sigma(I_{\alpha_1}^1)=I_{\alpha_1}^2,\ \sigma(I_{\alpha_2}^1)=I_{\alpha_2}^1$}
It can easily be seen that in this case, $\sigma(I_{\alpha_1}^2)=I_{\alpha_1}^1$ and $\sigma(I_{\alpha_2}^2)=I_{\alpha_2}^2.$ Therefore, $I_{\alpha_2}^1,I_{\alpha_2}^2,I_{\alpha_1}^3,I_{\alpha_2}^3\subset \tilde{C}_1$ and $I_{\alpha_1}^1,I_{\alpha_1}^2\subset\tilde{C}_2$, and hence
$$
Sep_{\mathcal{A}_S}^n(\mathbb{R})=\begin{cases}Sep_{\mathcal{A}}^n(\mathbb{R})\bigcup\left(I_{\alpha_1}^1\cup I_{\alpha_1}^2\right)& \text{ if $n$ is odd}\\
Sep_{\mathcal{A}}^n(\mathbb{R})& \text{ if $n$ is even}\end{cases}
$$
and the commutant is given by
$$
\mathcal{A}_S'=\mathcal{A}'\setminus \left\{\sum_{n\in \mathbb{Z}}f_n\delta^n\ |\ f_{2n+1}\neq 0 \text{ on }I_{\alpha_1}^1\cup I_{\alpha_1}^2\right\}
$$
Similar results can be obtained for the following cases.
\begin{enumerate}
\item $\sigma(s_i)=s_i$ for all $i=1,2$, $\sigma(I_{\alpha_1}^1)=I_{\alpha_1}^1\ (\Rightarrow \sigma(I_{\alpha_1}^2)=I_{\alpha_1}^2)$ and $\sigma(I_{\alpha_2}^1)=I_{\alpha_2}^2\ (\Rightarrow \sigma(I_{\alpha_2}^2)=I_{\alpha_2}^1).$
\item $\sigma(s_i)=s_i$ for all $i=1,2$, $\sigma(I_{\alpha_1}^1)=I_{\alpha_1}^2\ (\Rightarrow \sigma(I_{\alpha_1}^2)=I_{\alpha_1}^1)$ and $\sigma(I_{\alpha_2}^1)=I_{\alpha_2}^2\ (\Rightarrow \sigma(I_{\alpha_2}^2)=I_{\alpha_2}^1).$
\end{enumerate}
\subsubsection{$\sigma(s_1)=s_2\ (\Rightarrow \sigma(s_2)=s_1)$}
This is only true if $\sigma(I_{\alpha_1})=I_{\alpha_2}$ (and hence $\sigma(I_{\alpha_2})=I_{\alpha_1}$). This implies that $I_{\alpha_1},I_{\alpha_2}\subset C_2$ and all the new subintervals belong either to $\tilde{C}_2$ or to $\tilde{C}_4$ as can be seen in the two cases below.
\begin{enumerate}
\item $\sigma(I_{\alpha_1}^1)=I_{\alpha_2}^1$ and $\sigma(I_{\alpha_2}^1)=I_{\alpha_1}^1.$\\
This implies that $\sigma(I_{\alpha_1}^2)=I_{\alpha_2}^2$ and $\sigma(I_{\alpha_2}^2)=I_{\alpha_1}^2.$ Therefore all the new subintervals belong to $\tilde{C}_2.$ Therefore $I_{\alpha_i}^j\subset Sep_{\mathcal{A}_S}^n(\mathbb{R})$ for any odd $n$ and for all $i=1,2$ and all $j=1,2,3.$ Since $I_{\alpha_1},I_{\alpha_2}\subset Sep_{\mathcal{A}}^n(\mathbb{R})$ for any off $n\in \mathbb{Z},$ then, for any $n\in \mathbb{Z}$
$$Sep_{\mathcal{A}_S}^n(\mathbb{R})=Sep_{\mathcal{A}}^n(\mathbb{R})$$ and the commutants are the same in this case.
\item $\sigma(I_{\alpha_1}^1)=I_{\alpha_2}^1$, $\sigma(I_{\alpha_2}^2)=I_{\alpha_1}^2,\sigma(I_{\alpha_1}^2)=I_{\alpha_2}^2$ and $\sigma(I_{\alpha_2}^2)=I_{\alpha_1}^1.$ This implies that the new subintervals are mapped cyclically onto each other and hence all of them belong to $\tilde{C}_4.$ Therefore
$$
Sep_{\mathcal{A}_S}^n(\mathbb{R})=\begin{cases}Sep_{\mathcal{A}}^n(\mathbb{R})\bigcup\left(\cup_{i,j}I_{\alpha_i}^j\right)& \text{ if $4\nmid n$ }\\
Sep_{\mathcal{A}}^n(\mathbb{R})& \text{ if $4\mid n$ }\end{cases}
$$
and the commutant is given by
$$
\mathcal{A}_S'=\mathcal{A}'\setminus \left\{\sum_{n\in \mathbb{Z}}f_n\delta^n\ |\ f_{n}\neq 0 \text{ on }\cup_{i,j}I_{\alpha_i}^j \text{ for some $n$ such that }4\nmid n\right\}
$$
\end{enumerate}
\begin{remark}
\end{remark}
From the example of adding two jump points above, it can be seen that quite many cases to consider arise even by adding a small number of jump points. Taking a close look at the example reveals that there are $6$ distinct cases when two jump points are added (cases \ref{case1TRS} and \ref{Case2TRS} are the same).\par
If we let $p(n)$ denote the number of partitions of a positive integer $n$, then  the number of ways of distributing $n$ jump points among intervals is $p(n)$ (assuming sufficiently many intervals).
The case when we add $k$ jump points into a $C_1$ interval in turn gives rise to $p(k)p(k+1)$ sub cases to consider. Clearly, these are too many cases to write down, even for a small number of jump points added.

\section{Comparison of commutants for general sets}
Let $X$ be any set, $J$ a countable set and $\mathbb{P}=\{X_i\ :\  i\in J\}$ a partition of $X,$ that is $X=\bigcup\limits_{r\in J}X_r$ where $X_r\neq \emptyset$ for all $r\in J$ and $X_r\cap X_{r'}=\emptyset$ for $r\neq r'.$ Let $\mathcal{A}$ be the algebra of piecewise constant complex-valued functions and let $\sigma:X\to X$ be a bijection. The following lemma, whose proof can be found in \cite{AlexTRS}, gives the conditions under which $\mathcal{A}$ is invariant under $\sigma$ (and $\sigma^{-1}$).
\begin{lemma}\label{lem5TRS}
The following properties are equivalent.
\begin{enumerate}
\item The algebra $\mathcal{A}$ is invariant under $\sigma$ and $\sigma^{-1}.$
\item For every $i\in J$ there exists  $j\in J$ such that $\sigma(X_i)=X_j.$
\end{enumerate}
\end{lemma}
Let $\mathcal{A}$ be invariant under a bijection $\sigma:X\to X$, $\tilde{\sigma}:\mathcal{A}\to \mathcal{A}$ the automorphism induced by $\sigma$ and consider the crossed product algebra $\mathcal{A}\rtimes_{\tilde{\sigma}}\mathbb{Z}.$ It has been proven \cite{AlexTRS} that the commutant $\mathcal{A}'$ of the algebra $\mathcal{A}$ in the crossed product algebra $\mathcal{A}\rtimes_{\tilde{\sigma}}\mathbb{Z}$ is given precisely by
\begin{equation}\label{eqn18TRS}
\mathcal{A}'=\left\{\sum_{n\in \mathbb{Z}}f_n\delta^n\ :\ f_n=0\text{ on $C_k$ for some $n,k$ such that }\ :\ k\nmid n\right\},
\end{equation} where
\begin{align}\label{eqn19TRS}
C_k& :=\left\{x\in X~|~k\mbox{ is the smallest positive integer such that  } \exists\ X_i,\text{ such that } x,\sigma^k(x)\in X_i\right.\nonumber \\
&\qquad \left. {}\text{ for some }i\in J\right\}.
\end{align}
Now, suppose each of the partition sets $X_i$ is sub-partitioned into a finite disjoint union of its subsets, that is,
$X_i=\bigcup\limits_{r=1}^{s_i}X_{ir}$ where  $\emptyset \neq X_{ir}\subset X_i$ for each $X_{ir}$ and $X_{ir}\cap X_{ir'}=\emptyset$ if $r\neq r'$. Let $\mathcal{A}_S$ denote the algebra of piecewise constant functions on the new partitions. It can easily be seen that $\mathcal{A}$ is a subalgebra of $\mathcal{A}_S.$ We would like to compare the commutants $\mathcal{A}'$ and $\mathcal{A}_S'$ in the crossed product algebras $\mathcal{A}\rtimes_{\tilde{\sigma}}\mathbb{Z}$ and $\mathcal{A}_S\rtimes_{\tilde{\sigma}}\mathbb{Z}$ respectively, for a bijection $\sigma:X\to X.$ To do this, we must have that both $\mathcal{A}$ and $\mathcal{A}_S$ are invariant under $\sigma.$ We give the conditions in the following Lemma.
\begin{lemma}\label{lem6TRS}
Let $\mathcal{A}$ and $\mathcal{A}_S$ be as described above and let $\sigma:X\to X$ be a bijection on $X$ such that both $\mathcal{A}$ and $\mathcal{A}_S$ are invariant under $\sigma.$ If $X_i,X_j\in \mathbb{P}$ such that $X_i=\bigcup\limits_{k=1}^{s_i}X_{ik}$ and $X_i=\bigcup\limits_{l=1}^{s_j}X_{jl},$ and $\sigma(X_i)=X_j,$ then $s_i=s_j.$
\end{lemma}
\begin{proof}
Since $\sigma(X_i)=X_j$ and $\mathcal{A}_S$ is invariant under $\sigma$, then each set $X_{ik}$ in a partition of $X_i$ is mapped bijectively to a set, say $X_{jk}$ in the partition of $X_j.$ Since $\sigma$ maps $X_i$ bijectively to $X_j$, the the number of sets in the partition for $X_i$ must be the same as the number of sets in the partition for $X_j.$
\end{proof}
In what follows we make a comparison of the commutants the algebras $\mathcal{A}$ and $\mathcal{A}_S.$ We shall consider the following cases.
\begin{enumerate}
\item Only one of the partition sets, say $X_i,$ is sub-partitioned into a union of, say, $s$ subsets, that is,
$$
X_i=\bigcup_{j=1}^sX_{ij}
$$  which corresponds to adding a finite number of jump points in one partitioning interval of the real line.
\item A finite number of partition sets $X_1,X_2,\cdots,X_k\subset C_k$ are each partitioned into a union of, say, $s$ subsets, that is
$$
X_i=\bigcup_{j=1}^sX_{ij} \text{ for each }i=1,2,\cdots,k.
$$
This corresponds to adding jump points into different intervals on the real line.
\end{enumerate}
\subsection{Partitioning one set}
Suppose a set, say $X_0,$ is partitioned into a finite union of, say  $s$ subsets, that is
$$X_0=\bigcup_{j=1}^sX_{0j}$$ where $X_{0j}\neq \emptyset$ for all $j=1,\cdots,s$ and $X_{0j}\cap X_{0j'}=\emptyset$ if $j\neq j'.$
Then by Lemma \ref{finitep}, $\sigma(X_0)=X_0,$ that is $X_0\subset C_1$ where $C_k$ is defined by \eqref{eqn19TRS}.\par
Now, let
\begin{align}\label{eqn20TRS}
\tilde{C}_k& :=\left\{x\in X_0~|~k\mbox{ is the smallest positive integer such that  } \exists\ X_{0j},\text{ such that }\right.\nonumber \\
&\qquad \left. {} x,\sigma^k(x)\in X_{0j}\text{ for some }j\in \{1,\cdots,s\}\right\}.
\end{align}
Then each $X_{0j}\subset \tilde{C}_k$ for some $k\in\{1,\cdots,s\}.$ Therefore, for every $n\in \mathbb{Z},$
$$
Sep_{\mathcal{A}_S}^n(\mathbb{R})=Sep_{\mathcal{A}}^n(\mathbb{R})\bigcup\left(\cup_{k\nmid n}\tilde{C}_k\right)
$$
and the comparison of the commutants is given by;
$$
\mathcal{A}_S'=\mathcal{A}'\setminus\left\{\sum_{n\in \mathbb{Z}}f_n\delta^n\ :\ f_n\neq 0\text{ on }\tilde{C}_k\text{ for some $n,k$ such that }k\nmid n\right\}.
$$
\subsection{Partitioning more than one set}
Take sets $X_1,\cdots,X_K\subset C_k$, that is $\sigma^k(X_i)=X_i$ for each $i=1,\cdots,k.$ Since these sets are mapped bijectively onto each other by $\sigma,$ Lemma \ref{lem6TRS} implies that each of these sets must be partitioned as a union of the same number of subsets, that is
$$
X_i=\bigcup_{j=1}^sX_{ij} \text{ for each } i=1,\cdots,s.
$$
In the following Theorem, we give the comparison of the commutants.
\begin{theorem}
Suppose the sets $X_1,\cdots,X_K\subset C_k$ are each partitioned into a finite union of subsets as described above and $\sigma:X\to X$ is a bijection such that both $\mathcal{A}$ and $\mathcal{A}_S$ are invariant under $\sigma.$ Then
$$
\mathcal{A}_S'=\mathcal{A}'\setminus\left\{\sum_{n\in \mathbb{Z}}f_n\delta^n\ :\ f_n\neq 0\text{ on }\tilde{C}_{kl}\text{ for some $n,k,l$ such that }l\nmid {n\over k} \right\}.
$$
where
\begin{align}\label{eqn21TRS}
\tilde{C}_k& :=\left\{x\in X_i~|~k\mbox{ is the smallest positive integer such that  } \exists\ X_{ij},\text{ such that }\right.\nonumber \\
&\qquad \left. {} x,\sigma^k(x)\in X_{ij}\text{ for some }j\in \{1,\cdots,s\}\right\}.
\end{align}
\end{theorem}
\begin{proof}
Recall that the commutant $\mathcal{A}_S'$ is given by
$$
\mathcal{A}_S'=\left\{\sum_{n\in \mathbb{Z}}f_n\delta^n\ :\ f_n=0 \text{ on }Sep_{\mathcal{A}_S}^n(\mathbb{R})\right\}.
$$
By invariance of $\mathcal{A}_S$ under $\sigma,~\sigma$ maps the sets $X_{ij}$ bijectively onto each other. Since each $X_i\subset C_k$ for each $i=1,\cdots,k$, then each $X_{ij}\subset \tilde{C}_{kl}$ for some $l\in\{1,\cdots,s\},$ where $\tilde{C}_k$ is given by \eqref{eqn21TRS}.\par
Observe that $Sep_{\mathcal{A}_S}^n(\mathbb{R})=Sep_{\mathcal{A}}^n(\mathbb{R})$ for all $n$ such that $k\nmid n$ and for those $n$ such that $k\mid n,$ we have
$$
Sep_{\mathcal{A}_S}^n(\mathbb{R})=Sep_{\mathcal{A}}^n(\mathbb{R})\bigcup \left(\bigcup_{l=1,l\nmid {n\over k}}^s\tilde{C}_{kl}\right).
$$
Therefore, the comparison of the commutants is given by
$$
\mathcal{A}_S'=\mathcal{A}'\setminus\left\{\sum_{n\in \mathbb{Z}}f_n\delta^n\ :\ f_n\neq 0\text{ on }\tilde{C}_{kl}\text{ for some $n,k,l$ such that }l\nmid {n\over k} \right\}.
$$
\end{proof}
\begin{remark}
\end{remark}
\begin{enumerate}
\item For piecewise constant functions on the real line, adding $s$ jump points into one or more intervals corresponds to partitioning the interval/intervals into $2s+1$ sub-intervals (recall that we consider jump points to be intervals of zero length). Since we demand that jump points are mapped to jump points, each of the new sub-intervals belongs to $\tilde{C}_{kl}$ for some $l\in \{1,\cdots,s+1\}.$ For example, we do not have any new sub-intervals in say, $\tilde{C}_{k(2s+1)}.$ However, if we partition each of the sets $X_1,\cdots,X_k\subset C_k$ into a union of $2s+1$ subsets it's possible to have some of the new subsets in $\tilde{C}_{k(2s+1)}$ (if $\sigma$ maps the new subsets cyclically).
\item Also, in the general sets case, there is a possibility of having intervals in $C_{\infty},$ where  by $C_{\infty}$ we mean
$$
C_{\infty}:=\{x\in X\ :\ \not\exists\ j\in J\text{ such that }x,\sigma^k(x)\in X_j\text{ for all }k\geqslant 1\}.
$$
However, if two sets, say $X_i,X_r\subset C_{\infty}$ are each partitioned into a union of say $s$ subsets, that is
$$
X_i=\bigcup_{j=1}^s X_{ij}\text{ and } X_{r}=\bigcup_{j=1}^s X_{rj},
$$
then each of the new subsets $X_{ij},X_{rj}$ belong to $C_{\infty}$ and hence do not contribute anything new to the commutant.
\end{enumerate}

\par
\subsubsection*{Acknowledgement}
This research was supported by the Swedish International Development Cooperation Agency (Sida), International Science Programme (ISP) in Mathematical Sciences (IPMS), Eastern Africa Universities Mathematics Programme (EAUMP).  Alex Behakanira Tumwesigye is also grateful to the research environment Mathematics and Applied Mathematics (MAM), Division of Applied Mathematics, M\"alardalen University for providing an excellent and inspiring environment for research education and research.


\end{document}